\documentclass{article}
\usepackage[utf8]{inputenc}
\usepackage[table]{xcolor}
\usepackage{fullpage}
\usepackage{amsmath}         
\usepackage{amsfonts} 
\usepackage{graphicx} 
\usepackage{mathrsfs}
\usepackage{float}
\usepackage{enumitem}
\usepackage{float}
\usepackage{amsmath,amstext,amsthm,amssymb,amsxtra}
\usepackage[top=1.5in, bottom=1.5in, left=1.25in, right=1.25in]    {geometry}
\usepackage{tikz}
\usetikzlibrary{arrows}
\usepackage{mathtools}
\usepackage{lineno}
\usepackage{hyperref}
\usepackage{amssymb}

\DeclareMathOperator{\geo}{geo}
\numberwithin{equation}{section}
\mathtoolsset{showonlyrefs,showmanualtags} 
 
\newtheorem{theorem}{Theorem}[section]
\newtheorem{lemma}[theorem]{Lemma}
\newtheorem{proposition}[theorem]{Proposition}
\newtheorem{corollary}[theorem]{Corollary}
\newtheorem{example}[theorem]{Example}

\newcommand{\junk}[1]{}

\title{An analysis of a fair division protocol for drawing legislative districts}
\author{
Jessica De Silva\thanks{California State University, Stanislaus, Turlock, CA, USA jdesilva1@csustan.edu} 
\and
Brady Gales\thanks{Wake Forest University, Winston-Salem, NC, USA, galesb15@wfu.edu} 
\and
Bryson Kagy\thanks{Georgia Institute of Technology, Atlanta, GA, USA, brysonkagy@gatech.edu} 
\and
David Offner\thanks{Westminster College, New Wilmington, PA, USA, offnerde@westminster.edu}
}

\begin{document}
\maketitle
\begin{abstract}
    Landau, Reid, and Yershov [A Fair Division Solution to the Problem of Redistricting, \textit{Social Choice and Welfare}, 2008]
    propose a protocol for drawing legislative districts based on a two player fair division process, where each player is entitled to draw the districts for a portion of the state. We call this the \textit{LRY protocol}.  Landau and Su [Fair Division and Redistricting, arXiv:1402.0862, 2014] propose a measure of the fairness of a state's districts called the \textit{geometric target}. In this paper we prove that the number of districts a party can win under the LRY protocol can be at most two fewer than their geometric target, assuming no geometric constraints on the districts, 
    and provide examples to prove this bound is tight. We also show that if the LRY protocol is applied on a state with geometric constraints, the result can be arbitrarily far from the geometric target. 
    
\end{abstract}
\section{Introduction}

\textit{Gerrymandering} is the act of drawing legislative districts to favor one group over another. The focus of this paper is one of the most prevalent forms of gerrymandering, \textit{Partisan Gerrymandering}, where a political party draws districts to favor their party over opposing parties. Gerrymandering has recently been the focus of a number of high-profile court cases, and has drawn the interest of many mathematicians, some of whom have provided expertise in court. For example, the work of Herschlag, Ravier, and Mattingly in \cite{2017arXiv170901596H} was used in a recent United States supreme court case, while Pegden \cite{PegdenSC} testified in front of the Pennsylvania supreme court about the Pennsylvania congressional map.


Mathematicians have attempted to create metrics that detect intentional gerrymandering. 
In \cite{Dunchinmetric}, Duchin discusses some of the approaches used and challenges faced in developing a metric for detection of gerrymandering. In \cite{diffmetrics}, Warrington analyzes and compares numerous metrics that have been created for detecting gerrymandering, such as the Efficiency Gap.  In this paper we consider a measure of fairness called the \textit{geometric target}, proposed by Landau and Su in \cite{LandauSuProtocol}.  The geometric target is defined as the average of the best and worst cases for a party's measure of success.  We discuss the geometric target in more detail in Section \ref{geosec}.

\textit{Fair division} is the question of how to divide an object or set of objects among parties so that all parties receive a portion that is considered fair by their own evaluation. An example of a fair division problem is the classic \textit{I cut you choose} protocol for two-person cake cutting, which is envy-free in that each person gets a piece that they value as at least half the total of the cake. 
Finding a fair protocol for redistricting of states can be thought of as a \textit{fair division} problem. 
For example, Pegden, Procaccia, and Yu in \cite{Partisandistrictingprotocol} analyze an \textit{I cut you freeze} protocol for districting and provide guarantees for number of districts with majority support for a given party and the ability to pack a given sub-population into districts.

The fair division protocol we analyze in this paper, which we refer to as the \textit{LRY protocol}, is a two-player discrete fair-division procedure for districting described by Landau, Reid, and Yershov in \cite{LRYprotocol}. In the LRY protocol the two parties are presented with a sequence of splits of a given state, and asked to submit preferences for which side of the split they would prefer to draw the districts, with the other party getting to draw districts on the other side. In the end, a split is chosen and one party draws the districts on one side, and the other party draws the districts on the other.  
In Section 2 we fully explain this protocol using our notation. In Theorem~\ref{mainthm}, we show that if there are no geometric constraints on how the district is drawn, the number of districts won by a party under the LRY protocol can differ from that party's geometric target by at most 2 and from a party's $k$-split geometric target (a variation of the geometric target which is meaningful for the LRY protocol) by at most $\frac{3}{2}$.  

In Section 3 we consider the case where there are geometric constraints on how districts may be drawn. We define some geometric constraints corresponding to contiguity and compactness. In this case, we show in Example~\ref{geodelta} that it is possible to create a situation where the LRY protocol can return a result for a party where the number of districts they win is arbitrarily far from their geometric target.

\section{The LRY Protocol with no geometric constraints}

The LRY protocol was described by Landau, Reid, and Yershov in \cite{LRYprotocol}.  Here we summarize the protocol in our notation. 

\subsection{Notation for parties, support and splits}
We suppose there are two parties, denoted $A$ and $B$. We will use $P$ to denote a generic party, $P\in \{A,B\}$ and $\overline{P}$ to denote the party opposing $P$.  The goal of the protocol is to draw districts for a state, and we use $n \in \mathbb{N}$ to represent the number of districts to be drawn in the state, where each district contains the same number of people.  For a party $P$, let $x_P$ denote the total support of player $P$ in the state, and we assume $x_A+x_B=n$, so since there are $n$ districts, the support of the parties in each district is normalized to sum to 1.  

For $0 \le k \le n$ define a \textit{$k$-split} to be a division of the state so that  on one side of the split (by convention, we call this the left side) the two parties' support sums to $k$, i.e. there are $k$ districts' worth of population on the left side of the $k$-split, and $n-k$ districts' worth on the right side. For a $k$-split, let $L_k$ denote the area on the left of the split, and $R_k$ denote the area on the right. We denote by $S$ a side $S \in \{L,R\}$. Given a $k$-split, Let $S_k$ denote the area on the specified side of the $k$-split, and $\overline{S_k}$ denote the area on the other side.

Call a sequence of $k$-splits, $0 \le k \le n$ \textit{nested} if for every $k \ge 1$, $L_{k-1} \subseteq L_k$. For any nested sequence of $k$-splits, denote by $x_P(S_k)$ the total support for party $P$ in $S_k$, and let $x_P(k)= x_P(L_k) - x_P(L_{k-1})$, i.e. $x_P(k)$ is the support for party $P$ between the $(k-1)$-split and the $k$-split. Note that for all $k$, $x_A(k)+x_B(k)=1.$ Refer to Figure 4 in \cite{LandauSuProtocol} for an example of a state with nested  1-, 2-, 3-, and 4-splits.

\junk{
\begin{figure}[H]
    \centering
    \begin{tikzpicture}[scale=1.2]
 \draw[thick] (0,0) rectangle (7.5,3.75);
\draw[thick]node at (0,4){$0$-split}; 
\draw[thick]node at (2.25,4){$3$-split}; 
\draw[thick]node at (3.75,4){$5$-split};
\draw[thick]node at (5.25,4){$7$-split}; 
\draw[thick]node at (7.5,4){$10$-split};
\draw[dotted] (0,3.75)--(0,0);
\draw[dotted] (0.75 ,3.75)--(0.75 ,0);
\draw[dotted] (1.50 ,3.75)--(1.50 ,0);
\draw[dotted] (2.25 ,3.75)--(2.25 ,0);
\draw[dotted] (3 ,3.75)--(3 ,0);
\draw[dotted] (3.75 ,3.75)--(3.75 ,0);
\draw[dotted] (4.50 ,3.75)--(4.50 ,0);
\draw[dotted] (5.25 ,3.75)--(5.25 ,0);
\draw[dotted] (6 ,3.75)--(6 ,0);
\draw[dotted] (6.75 ,3.75)--( 6.7,0);
\draw[dotted] (7.5 ,3.75)--(7.5 ,0);
\end{tikzpicture} 
    \caption{An example state with 10 districts. Shown in dotted lines are nested 0, 3, 5, 7, and 10-splits.
    NOTE: I think this figure is not helpful.  Maybe refer to figures in LRY paper?}
    \label{fig:my_label}
\end{figure}
}

We assume that whichever candidate has more support in a district will win the district, and we adopt the convention that for all $k$, $x_P(S_k)$ is not an integer multiple of $.5$.  Therefore, we assume a party can always win a district with a strict majority, and there is no need to consider tied districts.

\subsection{The LRY Protocol}\label{LRYsubsec}
In the LRY protocol, an administrator (someone not affiliated with either party) presents a sequence of nested $k$-splits.

For all $k$, each party indicates which of the following options they prefer:
 \begin{enumerate}
 \item[]  \textbf{Option 1:} Party $A$ districts $L_k$ and Party $B$ districts $R_k$
\item[]  \textbf{Option 2:} Party $B$ districts $L_k$ and Party $A$ districts $R_k$
  \end{enumerate}
A party may also indicate that they are indifferent to the two options. 

The outcome of the protocol is as follows:
 \begin{enumerate}
 \item[]  \textbf{Outcome 1:} If there exists some $k$ such that Parties $A$ and $B$ both prefer the same option, then a map is created using that assignment of areas to district.
\item[]  \textbf{Outcome 2:} Else if there is a $k$ such that one party is indifferent but the other is not, the preferences of the non-indifferent party are chosen.
 \item[]  \textbf{Outcome 3:}  Else if there exists a $k$ such that both parties are indifferent then one of Option $1$ and Option $2$ is randomly chosen for that $k$.
\item[]  \textbf{Outcome 4 (Coin flip scenario):} Else it must be the case that there exists a $k$ such that for the $(k-1)$-split Party $A$ prefers Option $2$ and Party $B$ prefers Option $1$ but for the $k$-split they switch their preferences (Note that both parties prefer to district the right side in a $0$-split and the left side in an $n$-split). We call this scenario the \textit{coin flip scenario}. 
In this case, the protocol randomly returns one of the following four options, which we call the \textit{coin flip options}: 
\begin{itemize}
\item Option 1 for the $(k-1)$-split
\item Option 2 for the $(k-1)$-split
\item Option 1 for the $k$-split
\item Option 2 for the $k$-split.
\end{itemize}
  \end{enumerate}

In this section, we assume there are no geometric constraints on district lines, so when districting a side, a party is free to distribute their support among the districts in any way they like. The following example (illustrated in Figure~\ref{coinflipfig}) shows a situation where a coin flip scenario might arise.

\begin{example}\label{2gap} 
Suppose there is a state with 10 districts, and the following levels of support for the 5- and 6- splits. Let 
\begin{itemize}
    \item $x_A(L_5) = 1.9$, $x_B(L_5) = 3.1$
    \item $x_A(6) = 0.9$, $x_B(6) = 0.1$
    \item $x_A(R_6) = 1.4$, $x_B(R_6) = 2.6$.
\end{itemize} 
\end{example}
We suppose each party's goal is to maximize the number of districts they win. We examine the possible options for the 5- and 6- splits in Example~\ref{2gap}. 
\begin{itemize}
    \item $5$-split, Option 1: With support 1.9, $A$ can claim a majority in 3 districts on the left, leaving 2 for $B$, while with support 2.7, $B$ can win all 5 districts on the right.  So $A$ wins 3 and $B$ wins 7.
    \item $5$-split, Option 2: With support 2.3, $A$ can claim a majority in 4 districts on the right, leaving 1 for $B$, while with support 3.1, $B$ can win all 5 districts on the left.  So $A$ wins 4 and $B$ wins 6.
    \item $6$-split, Option 1: With support 2.8, $A$ can claim a majority in 5 districts on the left, leaving 1 for $B$, while with support 2.6, $B$ can win all 4 districts on the right.  So $A$ wins 5 and $B$ wins 5.
    \item $6$-split, Option 2: With support 1.4, $A$ can claim a majority in 2 districts on the right, leaving 2 for $B$, while with support 3.2, $B$ can win all 6 districts on the left.  So $A$ wins 2 and $B$ wins 8.
\end{itemize}
Thus $A$ prefers Option 2 and $B$ prefers Option 1 for the 5-split and they switch preferences for the 6-split.

\begin{figure}
    \centering
   \begin{tikzpicture}[scale=1.4]
 \draw[thick] (0,0) rectangle (11,2)
 node at (2,1.4){$\textcolor{red} {x_A(L_5)=1.9}$}
 node at (2,.6) {$\textcolor{blue}{x_B(L_5)=3.1}$};
 \draw[ ultra thick] (4.5,2) -- (4.5,0)
 node at (5.5,1.4){$\textcolor{red} {x_A(6)=.9}$}
 node at (5.5,.6) {$\textcolor{blue}{x_B(6)=.1}$};
  \draw[ ultra thick] (6.5,2) -- (6.5,0)
 node at (9,1.4){$\textcolor{red} {x_A(R_6)=1.4}$}
 node at (9,.6) {$\textcolor{blue}{x_B(R_6)=2.6}$} 
node[above] at (4.5,2){$5$-split}
node[above] at (6.5,2){$6$-split}
node[above] at (2,2){$L_5$}
node[above] at (9,2){$R_6$};
\end{tikzpicture}
    \caption{An example where the coin flip scenario applies.}\label{coinflipfig}
\end{figure}

A \textit{voting model} is a prediction of how the people in the state will vote, and of course in the real world, each party may have their own private voting model. It is interesting to note that if both parties share the same voting model, then in the LRY protocol parties with the goal of maximizing districts won will always choose different options or be indifferent to the options for a given $k$-split.  It is likely in the real world that two parties would not have the same voting model or have different goals (e.g. protecting incumbents), so in that case it is possible the LRY protocol could return Outcomes 1 or 2.

\subsection{Optimal Strategies}
Throughout the paper, we assume that the goal of each party is to maximize the number of districts they win. In this section, we assume there are no geometric constraints on how parties may draw district lines other than those imposed by the $k$-splits in the protocol. Given these assumptions, we now describe optimal strategies for each player, and state precisely how many districts they can win given their support.

Define $P_i(S_k,P_j)$ to be the number of districts won by party $P_i$ on side $S_k$ when party $P_j$ draws districts on $S_k$.  For example $A(L_7,B)$ is the number of districts that party $A$ will win among the 7 districts in $L_7$ when $B$ draws the districts on that side. Let $P(S_k)$ be the total number of wins for $P$ when they district $S_k$ and $\overline{P}$ districts $\overline{S}_k$. That is, \[P(S_k)=P(S_k,P)+P(\overline{S}_{k}, \overline{P}).\] 
Since the number of districts in the state is the sum of the districts won  by both parties, 
\[P(S_k) +\overline{P}(\overline{S}_{k})=n.\] 
Let $|S_k|$ be the number of districts in $S_k$, i.e. $|L_k| = k$ and $|R_k| = n-k$.

Proposition~\ref{optimal1} describes the number of districts a party can win when they draw district lines on one side of a $k$-split, and Corollary~\ref{optimal2} describes how many districts $P$ will win if their opponent draws the districts on one side of a $k$-split. Informally, if a player has a majority on the side where they are drawing districts, their best strategy is to divide their support evenly in each district, thus having a majority in each one.  If they have a minority, they should win as many districts as possible with just over .5 of their support in each district.  In all propositions, we state the result in terms of a party $P$, but to make the proofs more readable, since the protocol is symmetric, we will frequently assume without loss of generality that $P=A$ in the proofs.

\begin{proposition} \label{optimal1}
   $P(S_k,P)=\min\{\lfloor2x_P(S_k)\rfloor,\ |S_k|\}$.
\end{proposition}
\begin{proof}
There are two cases, depending whether party $P$ has a majority on side $S_k$ or not.

Case 1: Suppose $x_P(S_k)>\frac{|S_k|}{2}$, and thus $|S_k|=\min\{\lfloor2x_P(S_k)\rfloor,\ |S_k|\}$. In this case, player $P$ can create $|S_k|$ districts where in each district their support is $\frac{x_P(S_k)}{|S_k|}>\frac{1}{2}$.  Thus they can win $|S_k|$ districts, and this is the best possible.

Case 2: Suppose $x_P(S_k)<\frac{|S_k|}{2}$, and thus $\lfloor2x_P(S_k)\rfloor = \min\{\lfloor2x_P(S_k)\rfloor,\ |S_k|\}$.
For any district $i$ to be a victory for $P$ the support for $P$ in the district must be at least 1/2. Thus player $P$ can create $\lfloor2x_P(S_k)\rfloor$ districts with support just over $0.5$, and can win these districts, but their remaining support is less than $.5$, so they cannot win any more.
\end{proof}

Before proving Corollary~\ref{optimal2}, we need to prove a property of floors and ceilings. 

\begin{proposition} \label{ceil floor}Assume $x,y \in \mathbb{R}$ and  $x+y=k$ with $k\in \mathbb{N}.$ Then \[\min\{\lfloor 2x \rfloor,k\}+\max\{\lceil y-x\rceil,0\}=k.\]
\begin{proof}
Since $y=k-x$,
 \begin{equation}
   \min\{\lfloor 2x \rfloor,k\}+\max\{\lceil y-x\rceil,0\}= \min\{\lfloor 2x \rfloor,k\}+\max\{\lceil k-2x\rceil,0\}.
\end{equation}
Case 1: Assume $x> \frac{k}{2}$. Then \[\min\{\lfloor 2x \rfloor,k\}+\max\{\lceil k-2x\rceil,0\}=k+0=k.\]
Case 2: Assume $x\leq\frac{k}{2}$. Then 
\[\min\{\lfloor 2x \rfloor,k\}+\max\{\lceil k-2x\rceil,0\}=\lfloor2x\rfloor+\lceil k-2x\rceil=\lfloor2x\rfloor+k+\lceil-2x\rceil.\] 
Since $\lceil -z\rceil=-\lfloor z\rfloor$ for all $z \in \mathbb{R}$,   this quantity is equal to
$\lfloor 2x\rfloor+k-\lfloor 2x \rfloor=k$.
\end{proof}

\end{proposition}

\begin{corollary} \label{optimal2}
    $P(S_k,\overline{P}) = \max\left\{\lceil x_P(S_k) - x_{\overline{P}}(S_k) \rceil,0 \right\}$.
\end{corollary}
\begin{proof}
Without loss of generality, assume $P=A$. Since $x_A(S_k) + x_{B}(S_k)=|S_k|$ and $|S_k|\in \mathbb{N}$, by Lemma~\ref{ceil floor} it follows that  $$ \min\{\lfloor2x_{B}(S_k)\rfloor,\ |S_k|\} +\max\left\{\lceil x_A(S_k) - x_{B}(S_k) \rceil,0 \right\}=|S_k|.$$
 Additionally, $$|S_k|= B(S_k,B)+A(S_k,B).$$ 
 By Proposition~\ref{optimal1}, $B(S_k,B) = \min\{\lfloor2x_{B}(S_k)\rfloor,\ |S_k|\}$, so
 $$A(S_k,B) = \max\left\{\lceil x_A(S_k) - x_{B}(S_k) \rceil,0 \right\}.$$
\end{proof}

\subsection{The Geometric Target} \label{geosec}
The geometric target provides a measure of fairness for a districting protocol \cite{LandauSuProtocol}. The geometric target for a party $P$, denoted $\geo(P)$ is the average of a party's best case scenario and worst case scenario. Since we assume each party's goal is to maximize the number of districts won, this corresponds to the average of the number of districts a party wins when they district the whole state and the number of districts a party wins when the opposing party districts the whole state. 

When discussing the LRY protocol, it is also useful to define for each $k$ the \textit{$k$-split geometric target} for a party $P$, denoted $\geo_k(P)$.  This quantity is the average of a party's best case scenario and worst case scenario, with the restriction that every district must lie on one side or the other of the $k$-split.


In Propositions~\ref{geotarg} and \ref{k-split} we give formulas for the geometric target and $k$-split geometric targets in terms of a party's support.  Then in Proposition~\ref{geogeok} we show that the geometric target cannot differ from any $k$-split geometric target by more than 1/2 for any value of $k$. To do this, we need a result about differences between floors and ceilings.
\begin{proposition} \label{less than 1}
Let $r$ and $s$ be positive real numbers, and $t=r+s$. Then \begin{enumerate}[label=\roman*.]
    \item\label{i<1} $|\lceil t\rceil-(\lceil r\rceil+\lceil s\rceil)|\leq 1$ ,
    \item\label{ii<1} $|\lceil t\rceil - (\lceil r \rceil+\lfloor s\rfloor)|\leq 1$,
    \item\label{iii<1} $|\lfloor t\rfloor - (\lceil r \rceil+\lfloor s\rfloor)|\leq 1$,
    \item\label{iv<1} $|\lfloor t\rfloor - (\lfloor r \rfloor+\lfloor s\rfloor)|\leq 1.$
\end{enumerate}
\end{proposition}
\begin{proof} 
Parts~\ref{i<1} and \ref{ii<1} follow from the observations that 
\begin{align}
    \lceil r\rceil+\lceil s\rceil & \ge \lceil r+s \rceil,\\
    \lceil r+s \rceil &\ge \lceil r\rceil+\lfloor s\rfloor, \text{ and}\\
    (\lceil r\rceil+\lceil s\rceil) - (\lceil r\rceil+\lfloor s\rfloor)|&\le 1.
\end{align}

Similarly, Parts~\ref{iii<1} and \ref{iv<1} follow from the observations that 
\begin{align}
    \lfloor r\rfloor+\lfloor s\rfloor & \le \lfloor r+s \rfloor,\\
    \lfloor r+s \rfloor &\le \lceil r\rceil+\lfloor s\rfloor, \text{ and}\\
    (\lceil r\rceil+\lfloor s\rfloor) - (\lfloor r\rfloor +\lfloor s\rfloor)|&\le 1.
\end{align}
\end{proof}

\begin{proposition} \label{geotarg}
If $x_P > n/2$, then $\geo(P) = \frac{\lceil 2(x_P)\rceil}{2}$.  If $x_P < n/2$, then $\geo(P) = \frac{\lfloor2x_P\rfloor}{2}$.

\end{proposition}
\begin{proof}
Without loss of generality, assume $P=A$.
In both cases, the party $A$ will win the most districts when it draws all districts, and will win the fewest when $B$ draws all the districts. Thus, by Propositions~\ref{optimal1} and  \ref{optimal2}, when $x_P > n/2$,

\[
\geo(A) =\frac{A(L_n) + A(L_0)}{2}
=\frac{n+\lceil x_A-x_B\rceil}{2}
=\frac{n+\lceil x_A-(n-x_A)\rceil}{2}
=\frac{\lceil 2(x_A)\rceil}{2}.
\]

Similarly, when $x_P < n/2$,

\[
\geo(A) =\frac{A(L_n) + A(L_0)}{2}
=\frac{\lfloor2x_A\rfloor+0}{2}
=\frac{\lfloor2x_A\rfloor}{2}.
\]
\end{proof}

\begin{proposition}\label{k-split}
For any $k$, $\geo_k(P)=\frac{P(L_k)+P(R_k)}{2}$. 
\end{proposition}
\begin{proof}
The party $P$ will win the most districts when it draws  districts on both sides of the $k$-split, and will win the fewest when $\overline{P}$ draws districts on both sides of the $k$-split. Thus, for any $k$- split,
\begin{align*}
\geo_k(P)&=\frac{P(L_k,P)+P(R_k,P)+P(L_k,\overline{P})+P(R_k,\overline{P})}{2}\\
&=\frac{P(L_k,P)+P(R_k,\overline{P})+P(L_k,\overline{P})+P(R_k,P)}{2}\\
&=\frac{P(L_k)+P(R_k)}{2}.
\end{align*}
\end{proof}
 
\begin{proposition} \label{geogeok}
For all $k$, $|\geo(P)-\geo_{k}(P)|\leq\frac{1}{2}$.
\end{proposition}

\begin{proof}
Without loss of generality, assume $P=A$.
There are 4 cases, depending whether $A$ has a majority of the support in the state, and on each side of the $k$-split.  Without loss of generality, we assume that if $A$ has a majority of support on one side of the $k$-split and a minority of support on the other side, that it has a majority on the left. We explain the first case in detail, and sketch the rest.

Case 1: Assume $A$ has a majority of support in the state, and on each side of the $k$ split, i.e. $x_A>\frac{n}{2}$, $x_A(L_k)>\frac{k}{2}$, and $x_A(R_k)>\frac{n-k}{2}$.  Then, using Proposition~\ref{k-split}, Proposition~\ref{optimal1} and Corollary~\ref{optimal2},
\begin{align*}
    \geo_k(A) &= \frac{A(L_k)+A(R_k)}{2}\\
    &=\frac{A(L_k,A)+ A(R_k,B)+A(L_k,B)+A(R_k,A)}{2}\\
    &=\frac{k + \lceil x_A(L_k)-x_B(L_k)\rceil +\lceil x_A(R_k)-x_B(R_k)\rceil + (n-k)}{2}\\
    &=\frac{k + \lceil2x_A(L_k)-k\rceil+\lceil2x_A(R_k)-(n-k)\rceil + (n-k)}{2}\\
    &=\frac{\lceil2x_A(L_k)\rceil+\lceil2x_A(R_k)\rceil}{2}.
\end{align*}
By Lemma ~\ref{geotarg},  $\geo(A)=\frac{\lceil 2(x_A)\rceil}{2}$.
Thus, 
\[|\geo(A)-\geo_{k}(A)|=\left|\frac{\lceil 2x_A\rceil}{2}-\frac{\lceil2x_A(L_k)\rceil+\lceil2x_A(R_k)\rceil}{2}\right|.\]

Since $x_A(L_k) + x_A(R_k) = x_A$,  Proposition~\ref{less than 1} Part~\ref{i<1} implies this quantity is at most $1/2$.\\
Case 2: Assume $x_A>\frac{n}{2}$, $x_A(L_k)>\frac{k}{2}$ and $x_A(R_k)<\frac{n-k}{2}$. Then 
\[
\geo_k(A)=\frac{\lceil 2x_A(L_k)\rceil+\lfloor2x_A(R_k)\rfloor}{2}
\hspace{.5in} \text{ and } \hspace{.5in}
\geo(A)=\frac{\lceil 2(x_A)\rceil}{2}.
\]
Case 3: Assume $x_A<\frac{n}{2}$, $x_A(L_k)>\frac{k}{2}$ and $x_A(R_k)<\frac{n-k}{2}.$  Then 
\[
\geo_k(A)=\frac{\lceil 2x_A(L_k)\rceil+\lfloor2x_A(R_k)\rfloor}{2}
\hspace{.5in} \text{ and } \hspace{.5in}
\geo(A)=\frac{\lfloor 2x_A\rfloor}{2}.
\]
Case 4: Assume  $x_A<\frac{n}{2}$, $x_A(L_k)<\frac{k}{2}$ and $x_A(R_k)<\frac{n-k}{2}.$   Then 
\[
\geo_k(A)=\frac{\lfloor 2x_A(L_k)\rfloor+\lfloor2x_A(R_k)\rfloor}{2}
\hspace{.5in} \text{ and } \hspace{.5in}
\geo(A)=\frac{\lfloor 2x_A\rfloor}{2}.
\]
In Cases 2, 3, and 4, the conclusion follows from Proposition~\ref{less than 1}, Parts~\ref{ii<1}, \ref{iii<1}, and \ref{iv<1}, respectively.
\end{proof}

 \subsection{Analyzing the fairness of the LRY Protocol}
 In this section, we prove that the number of districts a party wins under the LRY protocol is within two of its geometric target.  Proposition~\ref{k-split} states that the $k$-split geometric target for a party $P$  is the average of the number of districts won by $P$ under Options 1 and 2 in the protocol.  Thus in at least one of these options the number of districts won by $P$ is at least the $k$-split geometric target. This is what is referred to as the ``Good Choice Property'' in \cite{LRYprotocol} and \cite{LandauSuProtocol}. Thus if the protocol ends in outcomes 1, 2, or 3, the number of districts won by $P$ will be at least its $k$-split geometric target, and thus by Proposition~\ref{geogeok}, within 1/2 of the geometric target.  Thus we devote the rest of the subsection to analyzing Outcome 4, the coin flip scenario.
 
 Before we prove the theorem, we finish analyzing the coin flip scenario in Example~\ref{2gap}. Analyzing the 5- and 6-splits with Proposition~\ref{optimal1} and Corollary~\ref{optimal2}, we get 
 \begin{align*}
  A(R_5,A) = 4, \quad  & A(R_5,B) = 0, \quad  & A(L_5,A) = 3, \quad  & A(L_5,B) = 0,\\ 
 A(R_6,A) = 2, \quad  & A(R_6,B) = 0, \quad  & A(L_6,A) = 5, \quad  & A(L_6,B) = 0.
 \end{align*}
 Thus $A(R_5) = 4$, $A(L_5) = 3$, $A(R_6) = 2$, and $A(L_6) = 5$. Proposition~\ref{k-split} implies $\geo_5(A) = \geo_6(A) = 3.5$, and since $x_A=4.3$, by Proposition~\ref{geotarg}, $\geo(A) = \frac{1}{2}\lfloor 2 \cdot 4.3 \rfloor = 4$.

Note that in this example in one of the coin flip options, ($A(R_6)$), $A$ wins only 2 districts, while the geometric target is 4.  Thus the protocol gives $A$ two fewer districts than is fair as defined by the geometric target, and 1.5 fewer than the $k$-split geometric target. The goal of this subsection is to show that this is the farthest possible outcome from either target.

For the remainder of the section, assume we have a fixed sequence of $k$-splits in a given state. Before considering the coin flip scenario, we first prove some results on the possible differences in the number of districts won for consecutive $k$-splits.

\begin{proposition}\label{minup}
   Suppose $x_P(k) < .5$.  Then $0 \le P(L_k,P)-P(L_{k-1},P) \le 1$
\end{proposition}
\begin{proof}
Note that $P(L_k, P) = \min\{\lfloor2x_P(L_k)\rfloor, k\} = \min\{\lfloor2x_P(L_{k-1}) + 2x_P(k) \rfloor, k-1+1\}$.
Thus
\[P(L_k, P) = \min\{\lfloor2x_P(L_{k-1}) + 2x_P(k) \rfloor, k-1+1\}
\ge \min\{\lfloor2x_P(L_{k-1})  \rfloor, k-1\} = P(L_{k-1}, P),
\]
and
\[P(L_k, P) = \min\{\lfloor2x_P(L_{k-1}) + 2x_P(k) \rfloor, k-1+1\}
\le \min\{\lfloor2x_P(L_{k-1})  \rfloor, k-1\}+1 = P(L_{k-1}, P)+1,
\]
justifying the first and second inequalities, respectively.
\end{proof}

\begin{proposition}\label{majup}
   Suppose $x_P(k) > .5$.  Then $1 \le P(L_k,P)-P(L_{k-1},P) \le 2$
\end{proposition}
\begin{proof}
Note that $P(L_k, P) = \min\{\lfloor2x_P(L_k)\rfloor, k\} = \min\{\lfloor2x_P(L_{k-1}) + 2x_P(k) \rfloor, k-1+1\}$.
Thus
\[
P(L_k, P) = \min\{\lfloor2x_P(L_{k-1}) + 2x_P(k) \rfloor, k-1+1\}
\ge \min\{\lfloor2x_P(L_{k-1})  \rfloor, k-1\} +1 = P(L_{k-1}, P)+1,
\]
and
\[
P(L_k, P) = \min\{\lfloor2x_P(L_{k-1}) + 2x_P(k) \rfloor, k-1+1\}
\le \min\{\lfloor2x_P(L_{k-1})  \rfloor, k-1\}+2 = P(L_{k-1}, P)+2,
\]
justifying the first and second inequalities, respectively.
\end{proof}

\begin{proposition}\label{mindn}
   Suppose $x_P(k) < .5$.  Then $0 \le P(R_k,\overline{P})-P(R_{k-1},\overline{P}) \le 1$
\end{proposition}

\begin{proof}
Without loss of generality, assume $P=A$.
Since $x_A(k) < .5$, $x_B(k) > .5$. Interchanging the roles of $A$ and $B$ and left and right sides (note $R_{k-1}$ has one more district in it that $R_k$), Proposition~\ref{majup} implies
\[1 \le B(R_{k-1},B) - B(R_k,B) \le 2.
\]
Since $A(R_{k-1},B) = n-k+1 - B(R_{k-1},B)$, and $A(R_{k},B) = n-k - B(R_{k},B)$,
\begin{align*}
A(R_k,B)-A(R_{k-1},B) 
&= n-k - B(R_{k},B) - (n-k+1 - B(R_{k-1},B))\\
&=B(R_{k-1},B) - B(R_{k},B) - 1.
\end{align*}
Thus $0 \le A(R_k,B)-A(R_{k-1},B) \le 1$.
\end{proof}

\begin{proposition}\label{majdn}
   Suppose $x_P(k) > .5$.  Then $-1 \le P(R_k,\overline{P})-P(R_{k-1},\overline{P}) \le 0$.
\end{proposition}

\begin{proof}
Without loss of generality, assume $P=A$.
Since $x_A(k) > .5$, $x_B(k) < .5$. Interchanging the roles of $A$ and $B$ and left and right sides (note $R_{k-1}$ has one more district in it that $R_k$), Proposition~\ref{minup} implies
\[0 \le B(R_{k-1},B) - B(R_k,B) \le 1.
\]
Since $A(R_{k-1},B) = n-k+1 - B(R_{k-1},B)$, and $A(R_{k},B) = n-k - B(R_{k},B)$,
\begin{align*}
A(R_k,B)-A(R_{k-1},B) 
&= n-k - B(R_{k},B) - (n-k+1 - B(R_{k-1},B))\\
&=B(R_{k-1},B) - B(R_{k},B) - 1.
\end{align*}
Thus $-1 \le A(R_k,B)-A(R_{k-1},B) \le 0$.
\end{proof}

\begin{lemma}\label{1change}
For $1 \le k \le n$, $P(L_{k-1}) \leq P(L_k) \leq P(L_{k-1}) +2.$
\end{lemma}

\begin{proof}
Without loss of generality, assume $P=A$.
Suppose $x_A(k)>0.5$. Then, using Propositions~\ref{majup} and \ref{majdn},
\begin{align*}
A(L_k) &= A(L_k,A) + A(R_k,B)\\
&= A(L_k,A) + A(R_k,B) + (A(L_{k-1},A) - A(L_{k-1},A)) + (A(R_{k-1},B) - A(R_{k-1},B))\\
&=A(L_{k-1},A)+A(R_{k-1},B)+(A(L_k,A)-A(L_{k-1},A))+(A(R_k,B)-A(R_{k-1},B))\\
&\ge A(L_{k-1},A)+A(R_{k-1},B)+1-1\\
&=A(L_{k-1},A)+A(R_{k-1},B)\\
&= A(L_{k-1}).
\end{align*}

Also using Propositions~\ref{majup} and \ref{majdn},
\begin{align*}
A(L_k) &= A(L_k,A) + A(R_k,B)\\
&= A(L_k,A) + A(R_k,B) + (A(L_{k-1},A) - A(L_{k-1},A)) + (A(R_{k-1},B) - A(R_{k-1},B))\\
&=A(L_{k-1},A)+A(R_{k-1},B)+(A(L_k,A)-A(L_{k-1},A))+(A(R_k,B)-A(R_{k-1},B))\\
&\le A(L_{k-1},A)+A(R_{k-1},B)+2+0\\
&=A(L_{k-1},A)+A(R_{k-1},B)+2\\
&= A(L_{k-1})+2.
\end{align*}

Suppose $x_A(k)<0.5$. Then, using Propositions~\ref{minup} and \ref{mindn},
\begin{align*}
A(L_k) &= A(L_k,A) + A(R_k,B)\\
&= A(L_k,A) + A(R_k,B) + (A(L_{k-1},A) - A(L_{k-1},A)) + (A(R_{k-1},B) - A(R_{k-1},B))\\
&=A(L_{k-1},A)+A(R_{k-1},B)+(A(L_k,A)-A(L_{k-1},A))+(A(R_k,B)-A(R_{k-1},B))\\
&\ge A(L_{k-1},A)+A(R_{k-1},B)+0-0\\
&=A(L_{k-1},A)+A(R_{k-1},B)\\
&= A(L_{k-1}).
\end{align*}

Also, using Propositions~\ref{minup} and \ref{mindn},
\begin{align*}
A(L_k) &= A(L_k,A) + A(R_k,B)\\
&= A(L_k,A) + A(R_k,B) + (A(L_{k-1},A) - A(L_{k-1},A)) + (A(R_{k-1},B) - A(R_{k-1},B))\\
&=A(L_{k-1},A)+A(R_{k-1},B)+(A(L_k,A)-A(L_{k-1},A))+(A(R_k,B)-A(R_{k-1},B))\\
&\le A(L_{k-1},A)+A(R_{k-1},B)+1 + 1\\
&=A(L_{k-1},A)+A(R_{k-1},B)+2\\
&= A(L_{k-1})+2.
\end{align*}
\end{proof}

Interchanging the roles of the left and right sides in Lemma~\ref{1change}, we obtain the following corollary.

\begin{corollary}\label{1changecor}
 For $1 \le k \le n$, $P(R_{k}) \leq P(R_{k-1}) \leq P(R_{k}) +2.$
\end{corollary}

We now analyze the coin-flip scenario. If the protocol ends with Outcome 4, using the $(k-1)$- and $k$-splits, then party $P$ must prefer to district one side for the $(k-1)$-split, and the other for the $k$-split. If $P(L_{k-1}) > P(R_{k-1})$ and $P(L_k) < P(R_k)$, then by Corollary~\ref{1changecor} and Proposition~\ref{1change},
\[P(R_k) \le P(R_{k-1}) < P(L_{k-1}) \le P(L_k) < P(R_k),
\]
a contradiction. Thus for the remainder of the section we assume for both parties, $P(L_{k-1}) < P(R_{k-1})$ and $P(L_k) > P(R_k)$.

In Proposition~\ref{atmost3}, we show that the largest difference between any two outcomes in the coin flip scenario is at most three wins.

\begin{proposition}\label{atmost3}
    In a coin flip scenario involving a $(k-1)$- and $k$-split, $P(R_{k-1})  - P(L_{k-1}) \le 3$ and $P(L_{k})  - P(R_{k}) \le 3$.
\end{proposition}

\begin{proof}
By the properties of the coin flip, Lemma~\ref{1change}, and Corollary~\ref{1changecor}
\begin{align*}
    P(R_{k-1}) - P(L_{k-1}) &\le P(R_{k-1}) - P(L_{k})+2\\
    &< P(R_{k-1}) - P(R_{k})+2\\
    &\le 2+2\\
    &=4
\end{align*}

Similarly,
\begin{align*}
    P(L_{k}) - P(R_{k}) &\le P(L_{k-1}) +2 - P(R_{k})\\
    &< P(R_{k-1}) - P(R_{k})+2\\
    &\le 2+2\\
    &=4
\end{align*}
Since $P(R_{k-1}) - P(L_{k-1})$ and $P(L_{k}) - P(R_{k})$ are each integers less than 4, they can be at most 3.
\end{proof}

We now prove the main theorem. Example~\ref{2gap} shows that both bounds in the theorem are tight.

\begin{theorem}\label{mainthm}
 Suppose the LRY protocol ends in a coin flip scenario involving a $(k-1)$- and $k$-split. Then for all $S \in \{L,R\}$ and $i \in \{k-1,k\}$,
 \begin{enumerate}
    \item $|\geo_i(P)-P(S_i)|\leq \frac{3}{2}$, and
    \item $|\geo(P)-P(S_i)|\leq 2$.
\end{enumerate}
\end{theorem}

\begin{proof}
To prove Part 1, use Propositions~\ref{k-split} and \ref{atmost3}, and the fact that $S_i$ is $L_i$ or $R_i$:
\begin{align*}
\left|P(S_i) -\geo_i(P)\right| &=  \left|P(S_i) -\frac{P(L_i) + P(R_i)}{2}\right|\\
&=  \left|\frac{P(L_i) - P(R_i)}{2}\right|\\
&\le \frac{3}{2}.
\end{align*}

For Part 2, by the triangle inequality and Proposition~\ref{geogeok},
\begin{align*}
    |\geo(P)-P(S_i)|
    &=|\geo(P)+\geo_i(P)-\geo_i(P)-P(S_i)|\\
    &\leq |\geo_i(P)-P(S_i)|+|\geo(P)-\geo_i(P)|\\
    &\leq \frac{3}{2} + \frac{1}{2} = 2.
\end{align*}
\end{proof}

\section{The LRY Protocol with Geometric Constraints}
We now show that given some geometric constraints on the district shape, it is possible under the LRY protocol for a party to win a number of districts that is arbitrarily far from their geometric target. 

First we define our geometric constraints. Our state is an $m \times m$ square grid, where each square represents an indivisible area of population. Denote by $d \in \mathbb{N}$ the number of squares that constitutes the population of one district. For $1 \le i,j \le m$, let $a_{i,j}$ denote the fraction of support for party $A$ in the square in the $i$th row and $j$th column. By convention, we draw the first row at the top, and first column at the left, and when context is clear, we will abuse notation and use $a_{i,j}$ to represent the square itself.

Our geometric constraints on districts are that they must be contiguous and compact. For contiguity, we require a district must be simply connected, and for compactness, we require that each district fits inside a $z$ by $z$ square, where $z=\lfloor 2\sqrt{d}\rfloor$.

\begin{example}\label{geodelta}
Let $\Delta \in \mathbb{N}$, $m=20\Delta$ and $d=100$, so there are $400\Delta^2$ squares and $4\Delta^2$ districts, each of which has 100 unit squares and must fit in a $20 \times 20$ square. Let $a_{i,j} = 1$ if it meets one of the following conditions
\begin{enumerate}
    \item $i\equiv 1, 2, 3, 4$, or $5 \bmod 20$ and $1\le j \le 10$,
    \item $i\equiv 6 \bmod 20$ and $j=1$,
\end{enumerate}
and let $a_{i,j} = 0$ otherwise. See Figure~\ref{constraintdist} for a partial drawing of the state.
\end{example}

\begin{figure}
\begin{center}
\begin{tikzpicture}[scale=.18]
\draw[fill=red!40] (0,20) -- (10,20) -- (10,15) -- (1,15) -- (1,14) -- (0,14) -- cycle;
\draw (0,20) -- (20,20) -- (20,0) -- (0,0) -- cycle;
\draw[ultra thick] (0,20) -- (5,20) -- (5,0) -- (0,0) -- cycle;
\node at (-5,10) {$L_1$};

\draw (20,20) -- (40,20) -- (40,0) -- (20,0) -- cycle;

\node at (45,10) {$\ldots$};

\draw[fill=red!40] (0,0) -- (10,0) -- (10,-5) -- (1,-5) -- (1,-6) -- (0,-6) -- cycle;
\draw (0,0) -- (20,0) -- (20,-20) -- (0,-20) -- cycle;
\draw[ultra thick] (0,0) -- (5,0) -- (5,-20) -- (0,-20) -- cycle;
\node at (-5,-10) {$L_2 \setminus L_1$};

\draw (20,0) -- (40,0) -- (40,-20) -- (20,-20) -- cycle;

\node at (45,-10) {$\ldots$};

\node at (10,-25) {$\ldots$};
\node at (30,-25) {$\ldots$};

\draw[fill=red!40] (0,-30) -- (10,-30) -- (10,-35) -- (1,-35) -- (1,-36) -- (0,-36) -- cycle;
\draw (0,-30) -- (20,-30) -- (20,-50) -- (0,-50) -- cycle;
\draw[ultra thick] (0,-30) -- (10,-30) -- (10,-40) -- (0,-40) -- cycle;
\node at (-5,-35) {$L_\Delta \setminus L_{\Delta - 1}$};

\draw (20,-30) -- (40,-30) -- (40,-50) -- (20,-50) -- cycle;

\node at (45,-40) {$\ldots$};

\draw (0,20) -- (70,20) -- (70,-50) -- (0,-50) -- cycle;

\draw (50,20) -- (70,20) -- (70,0) -- (50,0) -- cycle;

\draw (50,0) -- (70,0) -- (70,-20) -- (50,-20) -- cycle;

\draw (50,-30) -- (70,-30) -- (70,-50) -- (50,-50) -- cycle;

\node at (60,-25) {$\ldots$};

\end{tikzpicture}
\end{center}
\caption{An illustration of part of the state in Example~\ref{geodelta}.  Each large square corresponds to a $20 \times 20$ piece of the grid.  The shaded areas correspond to the groups of 51 squares where Party $A$ has support 1 in each small square.  The support for Party $A$ in the rest of the state is 0. The areas $L_i \setminus L_{i-1}$ for $i = 1$, $2$, and $\Delta$ are outlined with thick lines.}\label{constraintdist}
\end{figure}
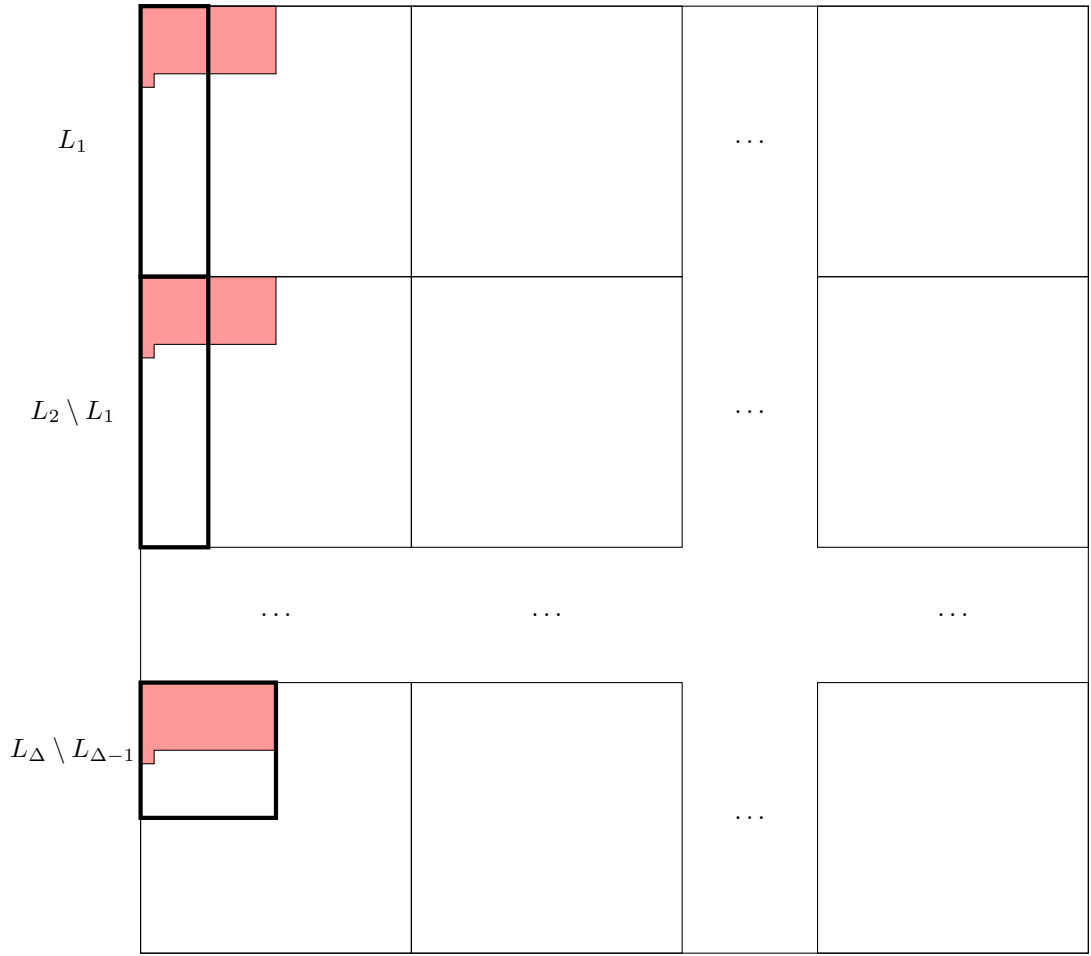

The total support for Party $A$ in Example~\ref{geodelta} is $51\Delta$. Since each district contains 100 squares, the support for $A$ in each square is either 1 or 0, and the squares are indivisible, the best possible outcome for Party $A$ is to win $\Delta$ districts.  This is in fact possible:  For $1 \le \ell \le \Delta$, if a district contains all squares with nonzero support for party $A$ in rows $1+ 20(\ell-1)$ through $20\ell$ then the support of $A$ in that district is 51.  Thus by drawing $\Delta$ such districts, and dividing the rest of the state arbitrarily, $A$ can win $\Delta$ districts.  It is also possible to draw districts so that $A$ will win none, and thus the geometric target for $A$ is $\Delta/2$. 

The key observation about Example~\ref{geodelta} is that $A$ can only win a district if that district contains all of the 1's in 20 consecutive rows in the first 10 columns. We now discuss the result of the LRY protocol on Example~\ref{geodelta}, where the splits will be chosen to break up these groups. For $0 \le k \le 4\Delta^2$, we specify the nested $k$-splits by describing  for $1 \le k \le 4\Delta^2$ the area that is in $L_k \setminus L_{k-1}$, i.e. the area in $L_k$ that is not in $L_{k-1}$. It is important to notice that these are not vertical $k$-splits like the ones in Figure 1.  For $1 \le k \le \Delta-1$, let  
\[L_k \setminus L_{k-1} = \{a_{i,j} : 1+20(k-1) \le i \le 20k, 1 \le j \le 5\},
\]
and let 
\[
L_\Delta\setminus L_{\Delta -1} = \{a_{i,j} : 1+20(\Delta-1) \le i \le 10+20(\Delta-1), 1 \le j \le 10\}.
\]
The key observation is that for $1 \le k \le \Delta-1$, $A(L_k)= 0$ and $A(R_k) = \Delta - k>0$, and $A(L_\Delta) = 1$ and $A(R_\Delta) = 0$.  Thus, regardless of the structure of the rest of the $k$-splits, if $k > \Delta$, $A(L_k) \ge 1$ and $A(R_k) = 0$, so we do not specify these splits explicitly. Note that for all of these splits, $B(S_k) = 4\Delta^2 - A(\overline{S}_k)$.

With these $k$-splits, for $0 \le k \le \Delta-1$, both parties prefer to district $R_k$, while for $\Delta \le k \le 4\Delta^2$, both parties prefer to district $L_k$. Thus the LRY protocol returns a coin flip scenario with the $(\Delta-1)$- and $\Delta$-splits. Note $A(L_{\Delta-1}) = 0$, $A(R_{\Delta-1}) = 1$, $A(L_{\Delta}) = 1$, and $A(R_{\Delta}) = 0$.  Thus two of the coin flip options are $\Delta/2$ less than $\geo(A)$.

Example~\ref{geodelta} implies the following theorem.

\begin{theorem}\label{arbitrary}
Under the LRY protocol with geometric constraints, the number of districts won by a party can be arbitrarily far from the geometric target.
\end{theorem}

We note that Example~\ref{geodelta} relies on a highly  structured pattern for a party's support, and $k$-splits that divide that pattern in a specific way.  Landau and Su \cite{LandauSuProtocol} note that by doing the LRY protocol repeatedly with many different patterns of $k$-splits, one is unlikely to have one party at a persistent disadvantage relative to their geometric target.

\section{Acknowledgements}
This research was conducted in Summer 2018  as a part of the Summer Undergraduate Applied Mathematics Institute at Carnegie Mellon University, funded by NSA grant number H98230-18-1-0043.


\end{document}